\documentclass{article}

\usepackage{arxiv}

\usepackage[utf8]{inputenc} % allow utf-8 input
\usepackage[T1]{fontenc}    % use 8-bit T1 fonts
\usepackage{hyperref}       % hyperlinks
\usepackage{url}            % simple URL typesetting
\usepackage{booktabs}       % professional-quality tables
\usepackage{amsfonts}       % blackboard math symbols
\usepackage{nicefrac}       % compact symbols for 1/2, etc.
\usepackage{microtype}      % microtypography
\usepackage{lipsum}
\usepackage{amsmath}

\usepackage{amsfonts,amssymb,stmaryrd}
\usepackage{graphicx}
\usepackage{amsthm}
\usepackage{xcolor}
\usepackage[shortlabels]{enumitem} % faz enumerate com letras ao invés de números

\newtheorem{theorem}{Theorem}[section]
\newtheorem{proposition}{Proposition}[section]
\newtheorem{lemma}{Lemma}[section]
\newtheorem{corollary}{Corollary}[section]

\newtheorem{definition}{Definition}[section]

\newcommand{\R}{\mathbb{R}}
\newcommand{\N}{\mathbb{N}}
\usepackage{subcaption}
\usepackage{cite}
\usepackage{authblk}
\newcommand{\p}{\partial}
\newcommand{\bb}{\begin{equation}}
\newcommand{\ee}{\end{equation}}
\newcommand{\ba}{\begin{array}}
\newcommand{\ea}{\end{array}}
\newcommand{\f}{\frac}
\usepackage[all]{xy}
\newcommand{\ds}{\displaystyle}

\newcommand{\al}{\alpha}
\newcommand{\be}{\beta}

\newcommand{\sign}{\text{sgn}\,}
\numberwithin{equation}{subsection}
\def\supp{\text{supp}\,}

\usepackage{ulem}

\title{Persistence and asymptotic analysis of solutions of nonlinear wave equations}

\author{
Igor~Leite~Freire$^{1,2}$\thanks{igor.leite.freire@gmail.com and igor.freire@ufscar.br}\\
$^1$ Institute of Advanced Studies,\\
Loughborough University\\
 LE11 3TU Epinal Way \\
Loughborough, United Kingdom,\\
$^2$Departamento de Matemática,\\
Universidade Federal de São Carlos,\\
Rodovia Washington Luís, Km 235, 13565-905,\\
São Carlos, SP - Brasil
}
\begin{document}
\maketitle

\begin{abstract}
We consider persistence properties of solutions for a generalised wave equation including vibration in elastic rods and shallow water models, such as the BBM, the Dai's, the Camassa-Holm, and the Dullin-Gottwald-Holm equations, as well as some recent shallow water equations with the Coriolis effect. We establish unique continuation results and exhibit asymptotic profiles for the solutions of the general class considered. From these results we prove the non-existence of non-trivial spatially compactly supported solutions for the equation. As an aftermath, we study the equations earlier mentioned in light of our results for the general class.
\end{abstract}

{\bf MSC classification 2010:} 35A01, 35Q35.


\textbf{Dedicatory:} This paper is dedicated to Professor Antonio Carlos Gilli Martins, who was an example of teacher, inspiration as a professional, and beloved friend. Rest in peace.

\newpage
\section{Introduction}\label{sec1}

In \cite{dai-acta} Dai deduced the following non-linear wave equation
\bb\label{1.0.1}
u_\tau+\sigma_1 u_{\tau\xi\xi}+\zeta_1 uu_\xi=-\sigma_2(2u_\xi u_{\xi\xi}+uu_{\xi\xi\xi}),
\ee
for describing finite length and amplitude waves propagating in Mooney-Rivlin materials, which encloses some polymeric elastomers, such as properly treated nature rubber, see \cite{dai-acta}. Above, $\tau=\epsilon t' $ and $\xi=x'-t'$, where $\epsilon$ is a small parameter; $t'$ and $x'$ denote dimensionless time and space variables; whereas $\sigma_1$ and $\sigma_2$ are negative parameters related to the material. 

Travelling waves for Dai's equation were later investigated in \cite{dai-wave,dai-proc} and, in particular, in \cite{dai-wave} it was shown that \eqref{1.0.1} has peakon solutions, which is a quite remarkable sort of solutions popularised after the famous work by Camassa and Holm \cite{chprl}, where the dispersive equation
\bb\label{1.0.2}
u_t-u_{txx}+\kappa u_x+3uu_x=2u_xu_{xx}+uu_{xxx},\quad \kappa\in\R,
\ee
named after them and referred henceforth as CH equation for short, was deduced using Hamiltonian methods in the study of shallow water regime. Peakons arise as weak solutions of \eqref{1.0.2} whenever $\kappa=0$.

Despite being deduced in rather different physical contexts (see the Introduction of \cite{dai-wave} for a nice discussion), mathematically speaking both equations \eqref{1.0.1} and \eqref{1.0.2} (even more when \eqref{1.0.2} is restricted to $\kappa=0$) are quite similar, and it would be expected that they share several mathematical properties. Very often Dai's equation is written as \bb\label{1.0.3}
u_t-u_{txx}+3uu_x=\gamma(2u_xu_{xx}+uu_{xxx}),
\ee
for some constant $\gamma$, see \cite{bran-cmp,guo-siam}, that can be obtained from \eqref{1.0.1} under the change of coordinates 
$$\tau=3\f{\sqrt{-\sigma_2}}{\zeta_1}t,\quad \xi=\sqrt{-\sigma_2}x.$$

While existence, uniqueness, and wave breaking of solutions for the CH equation can be found in the papers by Constantin and Escher \cite{const1998-1,const1998-2,const1998-3}, Constantin \cite{const2000-1}, and Rodriguez-Blanco \cite{blanco}, their counter-parts for the Dai's equation were reported by Brandolese \cite{bran-cmp}, Brandolese and Cortez \cite{bran-jfa}, and Guo and Zhou \cite{guo-siam}.

In \cite{guo-siam} the authors also considered persistence properties for Dai's equation by applying to it the ideas introduced by Himonas {\it et. al.} \cite{him-cmp} to tackle similar problem for the CH equation. As a consequence of these results, \eqref{1.0.3} cannot have compactly supported solutions at two different times. This fact was first noticed for the CH equation by Constantin \cite{const-jmp} and soon after Henry \cite{henry-jnmp} proved the same fact in a different way. 

Over the years, generalisations and extensions of both Dai's and CH equations have been proposed, some of them purely from a mathematical point of view, while others have been derived based on physical arguments. To name a few, we can mention:
\begin{itemize}
    \item the hyperelastic-rod wave like equation
    \bb\label{1.0.4}
    u_t-u_{txx}+\p_x\f{g(u)}{2}=\gamma(2u_xu_{xx}+uu_{xxx}),
    \ee
    introduced in \cite{co-siam};
    \item the {\it generalised rCH equation}, 
    \bb\label{1.0.5}
    u_t-u_{txx}+3uu_x=2u_xu_{xx}+uu_{xxx}+\al u_x+\be u^2u_x+\gamma u^3u_x+\Gamma u_{xxx},
    \ee
    considered in \cite{raspajde,raspasapm} (and \cite{freire-jde-2020-1} for its dissipative form). Such an equation was proposed as a mathematical generalisation of some physically derived models describing waves under the Coriolis effect proposed in \cite{chen-advances}, see also \cite{gui-jmfm,gui-jnl,chines-jde}.
    
\end{itemize}

One of our main interest and motivation is just \eqref{1.0.5}. Although it has received considerable attention\footnote{By considerable attention we do not mean just the equation studied in \cite{raspajde,raspasapm}, but actually those in \cite{chen-advances,gui-jmfm,gui-jnl,chines-jde}, which is enclosed in \eqref{1.0.5}}, apparently, to not say surprisingly, qualitative properties related with asymptotic behavior of its solutions seem not to have been considered yet. The same observation can also be extended to \eqref{1.0.4}.

The purpose of the present paper is to enlighten the aforementioned points for the models above. However, instead of treating them separately, we analyse the generalised hyperelastic-rod wave equation
    \bb\label{1.0.7}
    u_t-u_{txx}+\p_x\Big(f(u)+g(u)+\f{1}{2}f''(u)u_x^2\Big)-\p_x^3f(u)=0,
    \ee
which, as far as the author knows, was first considered by Holden and Raynauld \cite{holden-jde}, and latter studied by Brandolese and Cortez \cite{bran-jde}; and Tian, Yan, and Zhang in \cite{tian-nach}. Equation \eqref{1.0.7} clearly encloses all of the mentioned equations, among others. Moreover, by considering \eqref{1.0.7} we also complement, from a different perspective, the works \cite{holden-jde,bran-jde,tian-nach}.

In the next section we present the main notions and concepts needed for the paper, as well as our main results, the outline of the manuscript, its novelty and challenges. In section \ref{sec3} we revisit useful results needed to our demonstrations and certain technical, but very useful, propositions are proved. The machinery of section \ref{sec3} will be then used in section \ref{sec4} to demonstrate our main theorems, that are stated in section \ref{sec2}. We also apply our main results to the BBM, Dai's, DGH, and rCH equations in section \ref{sec5}. Our discussions are given in section \ref{sec6}, while our conclusions are made in section \ref{sec7}.

\section{Notation, important concepts, and main results}\label{sec2}

The $L^p(\R)$, $1\leq p\leq\infty$, and Sobolev spaces $H^s(\R)$, whose norms will be respectively denoted by $\|\cdot\|_{p}$ and $\|\cdot\|_{H^s(\R)}$, are the fundamental spaces of functions used henceforth. For any two functions $f$ and $g$, their convolution is denoted by $f\ast g$. 

For the CH equation and similar models, the variables $t$ and $x$ denote time and space, so that we maintain the same meaning for them throughout the manuscript. The derivatives of a function $u=u(t,x)$ with respect to the first variable is referred to as $u_t$, whereas $u_x$ or $\p_x u$ denote the derivatives with respect $x$. Also, whenever $I$ is an interval in $\R$ and $X$ is a Banach space, we write $u\in C^0(I;X)$ to say that $u(t,\cdot)\in X$ and $\|u\|=\sup\limits_{t\in X}\|u(t,\cdot)\|_X$, where $\|u(t,\cdot)\|_X$ denotes the norm of the function $x\mapsto u(t,x)$. In addition, $u\in C^k(I;X)$, for a natural number $k$, means that its derivatives with respect to $t$ belongs to $C^0(I;X)$ up to order $k$.

The solutions of \eqref{1.0.7} herein belong to $C^0([0,T];H^s(\R))$, $s>3/2$, for some $T>0$. For them we can invert the Helmholtz operator $1-\p_x^2$ and write \eqref{1.0.7} as a first order non-linear (and with a non-local term) evolution equation, obtaining
\bb\label{2.0.1}
u_t+f'(u)u_x+\p_x\Lambda^{-2}\Big(g(u)+\f{f''(u)}{2}u_x^2\Big)=0.
\ee

Given a function $h$ (for which the operations are defined), $\Lambda^{-2}h=p\ast h$ and 
$\p_x\Lambda^{-2}h=\p_x p\ast h$, where 
\bb\label{2.0.2}
p(x)=\f{e^{-|x|}}{2}
\ee
and $\p_x p(x)=-\sign{(x)}p(x)$, respectively. The last derivative has to be considered in the distributional sense.

A careful look at \eqref{2.0.1} tell us that the functions $f$ and $g$ cannot be arbitrary. It is then time to list what we request from them henceforward. Our conditions are
\begin{enumerate}
    \item[${\bf H}_1$] $f,g\in C^\infty(\R)$;
    \item[${\bf H}_2$] $f(0)=0$ and $f''(x)\geq0$, $x\in\R$;
    \item[${\bf H}_3$] $g(x)\geq 0$ and $g(x)=0$ if and only if $x=0$.
\end{enumerate}

Our first result is:
\begin{theorem}\label{teo2.1}
Let $u\in C^0([0,T];H^s(\R))$, $s>3/2$, be a solution for \eqref{2.0.1}. For each $t\in[0,T]$ fixed, we define
\bb\label{2.0.3}
h_t(x)=g(u(t,x))+\f{f''(u(t,x))}{2}u_x(t,x)^2,
\ee
and
\bb\label{2.0.4}
F_t(x)=(\p_x\Lambda^{-2} h_t)(x).
\ee

Moreover, assume that the conditions ${\bf H}_1$--${\bf H}_3$ hold. If we can find $t^\ast$, $a$, and $b$ such that $\{t^\ast\}\times[a,b]\subseteq(0,T)\times\R$, $h_{t^\ast}\big|_{(a,b)}\equiv0$, and $F_{t^\ast}(a)=F_{t^\ast}(b)$, then $u\equiv0$.
\end{theorem}

We observe that if $u$ is as described in theorem \ref{teo2.1}, then $u$ is a solution of \eqref{2.0.1} subject to the initial datum $u_0(x):=u(0,x)$. Thus \cite[Theorem 3.1]{tian-nach} implies the existence of a (local) unique solution $u\in C^0([0,T^\ast);H^s(\R))\cap C^1([0,T^\ast);H^{s-1}(\R))$, $s>3/2$, for some $T^\ast>0$. In particular, $0<T<T^\ast$.

It does not matter in theorem \ref{teo2.1} whether the solution is local or not. What is really relevant is the existence of a solution. As a short comment, even the question of uniqueness is somewhat unimportant for that result. 

It is worth mentioning that we restricted the first independent variable to a compact set $[0,T]$, which could be replaced by $[0,T)$ and the result would still be true. We, however, opted to maintain $[0,T]$ because some of our coming main results must necessarily have the restriction of $t$ to compact sets.

Theorem \ref{teo2.1} is a unique continuation result, which can be better explored in its consequences.

\begin{corollary}\label{cor2.1}
If ${\bf H}_1$--${\bf H}_3$ hold, $u\in C^0([0,T];H^s(\R))$, $s>3/2$, is a solution of \eqref{2.0.1}, and there exists a non-empty open set $\Omega\subseteq[0,T]$ for which $u$ vanishes, then $u$ is trivial.
\end{corollary}

The adjective trivial above refers to any function vanishing everywhere.

\begin{corollary}\label{cor2.2}
If ${\bf H}_1$--${\bf H}_3$ hold and $u\in C^0([0,T];H^s(\R))$, $s>3/2$, is a non-trivial solution of \eqref{2.0.1}, then we cannot find $t^\ast\in (0,T)$ and $[a,b]\subseteq\R$ such that $u(t^\ast,x)=0$, $x\in[a,b]$, and $u_t(t^\ast,a)=u_t(t^\ast,b)$.
\end{corollary}

\begin{corollary}\label{cor2.3}
If ${\bf H}_1$--${\bf H}_3$ hold, then none non-trivial solutions of the equation \eqref{2.0.1}
$u\in C^0([0,T^\ast];H^s(\R)),$ $s>3/2$, can be compactly supported on $[0,T]\times\R$.
\end{corollary}

Although corollary \ref{cor2.3} is a result of non-existence of compactly supported solutions for \eqref{2.0.1}, the existing literature suggests that it could be improved. In fact, for the CH equation is well known that given a solution $u$ in $C^0([0,T];H^s(\R))$, $s>3/2$, then it can be spatially compactly supported at most for a single value of $t$, see \cite{const-jmp,henry-jnmp,him-cmp,bran-imrn}.
In view of this fact we have our next theorem.

\begin{theorem}\label{teo2.2}
If ${\bf H}_1-{\bf H}_3$ hold, $f'(0)=0$, and let $u\in C^0([0,T];H^s(\R))$, $s>3/2$, be a solution of \eqref{2.0.1}. Then $u$ cannot be (spatially) compactly supported at two different times.
\end{theorem}

Our next result is a key ingredient for establishing the precedent theorem, since it describes the behavior of solutions $u$ of \eqref{2.0.1} with very strong decaying as $|x|\rightarrow\infty$. 

\begin{theorem}\label{teo2.3}
Assume that ${\bf H}_1-{\bf H}_3$ hold, 
\bb\label{2.0.5}
g(u(t,x))\leq cu(t,x)^2,
\ee
for some $c>0$, $f'(0)=0$, $d>1/2$, $s>3/2$, and let $u_0\in H^s(\R)$ be a non-trivial function satisfying
\bb\label{2.0.6}
\sup_{x\in\R}\Big(e^{\f{|x|}{2}}(1+|x|)^{1/2}(\ln{(1+|x|)})^d(|u_0(x)|+|u_0'(x)|)\Big)<\infty.
\ee
Suppose that $u=u(t,x)$ is the (unique) corresponding solution of \eqref{2.0.1} subject to $u(0,x)=u_0(x)$, and \eqref{2.0.5}. Then there exists a constant $K>0$ (depending only on $f$, $g$ and $T$) such that
\bb\label{2.0.7}
\sup_{x\in\R}\Big(e^{\f{|x|}{2}}(1+|x|)^{1/2}(\ln{(1+|x|)})^d(|u(t,x)|+|\p_xu(t,x)|)\Big)\leq K.
\ee

Moreover, there exists continuous functions $\Phi_\pm:[0,T]\rightarrow\R$, $\epsilon_\pm:[0,T]\times\R\rightarrow\R$, and $R:\R\rightarrow\R$, such that
\bb\label{2.0.8}
u(t,x)=u_0(x)\pm\, t\,e^{-|x|}\Big(\Phi_\pm(t)+\epsilon_\pm(t,x)+R(x)\Big),
\ee
where the sign $+$ is taken for $x>0$, whereas for $x<0$ we choose $-$. Moreover,
\begin{itemize}
    \item $\Phi_\pm$ is non-negative and $\Phi_\pm(t_0)=0$ if and only if $u(t_0,x)=0$;
    \item $\epsilon_\pm(t,x)\rightarrow0$ as $\pm x\rightarrow\infty$;
    \item $\Phi_\pm$ is uniformly bounded, that is, $c_1\leq \Phi_\pm(t)\leq c_2$, for some non-negative constants $c_1$ and $c_2$;
    \item $R(x)\sim O((\ln{(1+|x|)})^{-d}(1+|x|)^{-1})$.
\end{itemize}
\end{theorem}

We recall that $f_1(x)\sim O(f_2(x))$ as $x\nnearrow\infty$ if there exists some $L$ such that
$$
\lim_{x\rightarrow\infty}\f{f_1(x)}{f_2(x)}=L,
$$
whereas $f_1(x)\sim o(f_2(x))$ as $x\nnearrow\infty$ means that
$$
\lim_{x\rightarrow\infty}\f{f_1(x)}{f_2(x)}=0.
$$

A similar definition can be done for $x\ssearrow-\infty$.

We have two more theorems to complete our main results' list, but before stating them we need a couple of concepts. To tackle our problems we use ideas introduced by Brandolese \cite{bran-imrn}, so that it is convenient to recall some important notions.

A function $f:[a,b]\rightarrow\R$ is called absolutely continuous on $[a,b]$ if, for every $\varepsilon>0$, there exists $\delta>0$ such that
$$
\sum_{i=1}^N |f(b_i)-f(a_i)|<\varepsilon,
$$
for any finite collection of non-overlapping compact subsets $[a_i,b_i]\subseteq[a,b]$, $1\leq i\leq N$, with 
$$\sum_{i=1}^N(b_i-a_i)<\delta.$$

The collection of such functions is denoted by $AC[a,b]$. Moreover, a function $f:\R\rightarrow\R$ is said to be locally absolutely continuous if $f\in AC[a,b]$, for all $-\infty<a<b<\infty$.

Henceforth, it will always be assumed that by a weight (or weight function) $v:\R\rightarrow\R$ we mean a continuous and positive function. Additional conditions will be imposed in due course.

A weight function $v$ is called {\it sub-multiplicative} if $v(x+y)\leq v(x)v(y)$, for all $x,y\in\R$. Note that a necessary condition for a weight function $v$ to be sub-multiplicative is $v(0)\geq 1$.

A function $\phi:\R\rightarrow\R$ is said to be $v-$moderate if $\phi$ is positive and there exists a constant $c_0>0$ such that
\bb\label{2.0.9}
\phi(x+y)\leq c_0\,v(x)\,\phi(y).
\ee
In particular, we have $c_0\geq 1/v(0)$. If a function $\phi$ is $v-$moderate, for some function $v$, we simply say that $\phi$ is moderate.

\begin{definition}\label{def2.1}
An admissible weight function for the equation \eqref{2.0.1} is a locally absolutely continuous function $\phi:\R\rightarrow\R$, such that $|\phi'(x)|\leq A |\phi(x)|$ almost everywhere (a.e.), for some $A>0$, and $v-$moderate for some continuous, sub-multiplicative weight function $v$ satisfying $\inf_\R v>0$ and $e^{-|\cdot|}v(\cdot)\in L^1(\R)$.
\end{definition}

We observe that if $u\in C^0([0,T];H^s(\R))$, $s>3/2$, is a solution of \eqref{2.0.1}, then the quantity
\bb\label{2.0.10}
{\cal H}(t)=\int_\R(u(t,x)^2+u_x(t,x)^2)dx=\|u(t,\cdot)\|^2_{H^1(\R)}
\ee
is invariant, that is, for any $t\in(0,T]$, ${\cal H}(t)={\cal H}(0)=\|u_0\|^2_{H^1(\R)}$. Therefore, the Sobolev Embedding Theorem tells us that $\|u(t,\cdot)\|_\infty\leq \|u_0\|_{H^1(\R)}$, whereas the local well posedness result \cite[Theorem 3.1]{tian-nach} implies that $u_x(t,\cdot)\in L^\infty(\R)\cap C^0(\R)$. Hence, $\sup\limits_{t\in[0,T]}\|u_x(t,\cdot)\|_\infty$ is uniformly bounded, and from the conditions on $f$ we infer that the quantities
$$
\sup_{t\in[0,T]}\|f(u(t,\cdot))\|_\infty,\quad \sup_{t\in[0,T]}\|f' (u(t,\cdot))\|_\infty,\quad \sup_{t\in[0,T]}\|f''(u(t,\cdot))\|_\infty
$$
are equally uniformly bounded. As a result, we can find a constant $M>0$, independent of $t$, but eventually depending on $T$, $f$, and $\|u_0\|_{H^1(\R)}$, such that
\bb\label{2.0.11}
\ba{lcl}
M&=&\ds{\sup_{t\in[0,T]}\|u(t,\cdot)\|_\infty+\sup_{t\in[0,T]}\|u_x(t,\cdot)\|_\infty+\sup_{t\in[0,T]}\|f(u(t,\cdot))\|_\infty}\\
\\
&+&\ds{ \sup_{t\in[0,T]}\|f' (u(t,\cdot))\|_\infty+ \sup_{t\in[0,T]}\|f''(u(t,\cdot))\|_\infty}.
\ea
\ee

\begin{theorem}\label{teo2.4}
Assume that ${\bf H_1}-{\bf H_3}$ hold. Let $T>0$, $s>3/2$, and $u\in C^0([0,T];H^s(\R))$ be a solution of \eqref{2.0.1} subject to $u(0,x)=u_0(x)$. If $\phi$ is an admissible weight function for \eqref{2.0.1} (see definition \ref{def2.1}) and the initial datum satisfies the conditions
$$\phi u_0,\,\phi u_0'\,\in L^{p}(\R),\,\,2\leq p\leq\infty,$$
then
\bb\label{2.0.12}
\|\phi(\cdot)u(t,\cdot)\|_p+\|\phi(\cdot)u_x(t,\cdot)\|_p\leq \kappa\Big(\|\phi(\cdot)u_0(\cdot)\|_p+\|\phi(\cdot)u_0'(\cdot)\|_p\Big),
\ee
for some constant $\kappa>0$ depending on $T$, $A$, $c_0$, $\inf_\R v$, $\|e^{-|\cdot|}v(\cdot)\|_{1}$, $f$, $g$, and $M$.
\end{theorem}

Finally, we have
\begin{theorem}\label{teo2.5}
Assume that ${\bf H_1}-{\bf H_3}$ hold. Let $2\leq p\leq \infty$, $u_0\in H^s(\R)$, $s>3/2$, $\phi$ a $v-$moderate weight function such that $u_0$ and $\phi$ satisfy
\bb\label{2.0.13}
\phi^{1/2}u_0,\,\,\phi^{1/2}u_0' \in L^2(\R),
\ee
and
\bb\label{2.0.14}
\phi u_0,\,\,\phi u_0'\in L^p(\R).
\ee

Suppose that $e^{-|\cdot|}v(\cdot)\in L^p(\R)$, $2\leq p\leq \infty$, $g$ satisfies the condition \eqref{2.0.5}, and let $u$ be the corresponding solution of \eqref{2.0.1} subject to $u(0,x)=u_0(x)$. Then there exists a constant $K$, depending on $\phi$, $v$, $f$, $g$, $T$, and $u_0$ (and its $L^p(\R)-$norm), such that
\bb\label{2.0.15}
\sup_{t\in[0,T]}\|\phi(\cdot)^{1/2}(\cdot)u(t,\cdot)\|_2+\|\phi(\cdot)^{1/2}u_x(t,\cdot)\|_2\leq K,
\ee
for some positive constant $K$, and
\bb\label{2.0.16}
\sup_{t\in[0,T]}\|\phi(\cdot)u(t,\cdot)\|_p+\|\phi(\cdot)u_x(t,\cdot)\|_p< \infty.
\ee
\end{theorem}

Theorems \ref{teo2.4} and \ref{teo2.5} have results that may appear similar at a first glance, but a more careful look shows substantial difference. Indeed, the weight $\phi$ in theorem \ref{teo2.4} is $v-$moderate by a function $v$ such that $e^{-|\cdot|}v(\cdot)\in L^1(\R)$ (see definition \ref{def2.1}), whereas the weight $\phi$ in theorem \ref{teo2.5} is $v-$moderate by a function $v$ satisfying $e^{-|\cdot|}v(\cdot)\in L^p(\R)$, for any $p\geq2$.

{\bf Challenges and novelties of the manuscript.} The fact that we are considering \eqref{1.0.7}, or its non-local evolution form \eqref{2.0.1}, which has two arbitrary functions, brings some difficulties in view of the natural arbitrariness involved. For this reason, it is not surprising that we should somewhat restrict these functions in order to be able to address the problem. 

All of our results require the conditions ${\bf H}_1$ and ${\bf H}_2$ fulfilled, and most of them also require the additional restriction ${\bf H}_3$. The first two conditions are not new, in the sense that similar requirements were already used in earlier works dealing with problems related to existence of solutions, see \cite[Equation (1.6)]{holden-jde}, \cite[Theorem 2.1]{bran-jde}, or \cite[Theorem 5.1]{tian-nach}. What seems to be new in our case is just ${\bf H}_3$, which requires that $g$ is non-negative and the only solution of the equation $g(x)=0$ is just $x=0$. The latter condition cannot be relaxed since some of our demonstrations are based on the following idea: find a value of $t$, say $t^\ast$, for which $h_{t^\ast}(x)=0$, where $h_{t^\ast}(\cdot)$ is given by \eqref{2.0.3}. Under our conditions we can guarantee that $u(t^\ast,x)=0$, otherwise theorem \ref{teo2.1} and its corollaries could not be proved. 

The three aforesaid conditions are somewhat unsurprising, but the same cannot be said about the inequality \eqref{2.0.5} and the restriction $f'(0)=0$ in theorems \ref{teo2.4} and \ref{teo2.5}, which seem to be truly new conditions when dealing with \eqref{1.0.7}.

In the proof of theorem \ref{teo4.7} (see section \ref{sec4}) the condition $g(u)\leq c|u|^2$, for some positive constant $c$, is sine qua non. Such a theorem is an essential ingredient for the demonstration of theorem \ref{teo2.5} (see subsection \ref{sec4.3}), which is crucial to establish theorems \ref{teo2.3} and \ref{teo2.2}, see subsection \ref{sec4.4}. In lieu of \eqref{2.0.5} we could require $|g(x)|\leq c x^2$, which is somewhat stronger than the original requirement because this would be a condition on the function $g$, whereas \eqref{2.0.5} is, in fact, a condition on the composed function $g\circ u$. Moreover, the sort of solutions $u$ we are dealing with is bounded in view of the Sobolev Embedding Theorem, which makes easier to have \eqref{2.0.5} satisfied for a wider class of equations. 

As an example, let us consider the function
$$g(x)=\f{x^2}{2}+\f{\be}{3}x^3+\f{\gamma}{4}x^4,$$
with $\be>0$ and $\gamma\geq0$. It does not satisfy the condition $g(x)\leq cx^2$ for all $x\in\R$ because the term $x^4$ is dominant for higher values of $x$ (for $\gamma=0$ the leading term is $x^3$). On the other hand, if $u(t,\cdot)\in C^0([0,T];H^s(\R))$, with $s>3/2$, is a solution of \eqref{2.0.1}, then we have the estimate
$$
g(u(t,x))\leq \Big(\f{1}{2}+\f{\be}{3}\|u(0,\cdot)\|_{H^1(\R)}+\f{\gamma}{4}\|u(0,\cdot)\|_{H^1(\R)}^2\Big)u(t,x)^2.
$$

Above we used the fact that $u$ is bounded, $\|u(t,\cdot)\|_\infty\leq\|u(0,\cdot)\|_{H^1(\R)}$, and the functional \eqref{2.0.10} is invariant.

These observations will be recalled in section \ref{sec5}, where we apply our results to some specific models.

\section{Preliminaries results}\label{sec3}

In what follows, we denote the set of the positive integers by $\N$.

\begin{lemma}\label{lema3.1}
Assume that $f\in C^\infty(\R)$, $f(0)=0$, and $s>1/2$. For all $v\in H^s(\R)$, then $f(v)\in H^s(\R)$.
\end{lemma}

\begin{proof}
See \cite[Lemma 1]{const-mol} and references therein.
\end{proof}

\begin{lemma}\label{lema3.2}
Let $1\leq p\leq\infty$ and $v$ be a sub-multiplicative weight on $\R$. The following conditions are equivalent:
\begin{itemize}
    \item $\phi$ is a $v-$moderate weight function;
    \item for all measurable functions $f$ and $g$, the weighted Young estimate holds
    $$
    \|(f\ast g)\phi\|_p\leq c_0\|fv\|_1\|g\phi\|_p,
    $$
for some positive constant $c_0$.
\end{itemize}
\end{lemma}

\begin{proof}
See \cite[Proposition 3.2]{bran-imrn}.
\end{proof}

We now present some propositions that will be used when proving our main theorems. Parts of the proofs of propositions \ref{prop3.4}--\ref{prop3.6} can be found, or inferred from, the paper by Brandolese \cite{bran-imrn}. We, however, opt to present them for sake of completeness and clarity, since the powerful techniques introduced in \cite{bran-imrn} seem not to be widely used.

\begin{proposition}\label{prop3.1}
If $f\in C^\infty(\R)$, $n\in\N$, $s>3/2$, and $v\in H^s(\R)$, then $f(v)v_x^n\in H^{s-1}(\R)$.
\end{proposition}

\begin{proof}
Let us first consider $n=1$ and rewrite $f(v)v_x=(f(v)-f(0))v_x+f(0)v_x$. Clearly $f(0)v_x\in H^{s-1}(\R)$, while $f(\cdot)-f(0)$ satisfies the conditions in Lemma \ref{lema3.1}, and thus it belongs to $H^s(\R)$. By the algebra property (e.g., see \cite[page 320, exercice 6]{taylor}) we conclude that $(f(v)-f(0))v_x\in H^{s-1}(\R)$. The general case is a consequence of the algebra property and $f(v)v_x^n=(f(v)v_x)v_x^{n-1}$.
\end{proof}

\begin{proposition}\label{prop3.2}
If $f\in L^1(\R)\cap L^\infty(\R)$, then $f\in L^p(\R)\cap L^\infty(\R)$, for any $1\leq p\leq\infty$.
\end{proposition}

\begin{proof}
Let $a:=\max\{1,\|f\|_\infty$\}. Then $|f(x)|^p\leq a^{p-1} |f|$ and $\|f\|_p\leq a\|f\|_1$.
\end{proof}

\begin{proposition}\label{prop3.3}
Assume that $f\in L^1(\R)\cap L^\infty(\R)\cap C^0(\R)$ and $p>1$. Then, for any $\epsilon\in(0,\|f\|_\infty)$, there exists $L_\epsilon>0$ and $r\in(1,p)$ such that
$$
\sqrt[p]{L_\epsilon}(\|f\|_\infty-\epsilon)\leq\|f\|_p\leq\|f\|_r^{r/p}\|f\|_\infty^{1-r/p}.
$$
\end{proposition}

\begin{proof}
In view of proposition \ref{prop3.2} we know that $f\in L^r(\R)$, $r>1$. On the other hand, for $1\leq r<p$, we have
$$
\|f\|_p^p=\int_\R|f(x)|^{p-r}|f(x)|^rdx\leq \|f\|_\infty^{p-r}\|f\|_r^r.
$$

Let $\epsilon\in(0,\|f\|_\infty)$ and $X_\epsilon:=\{x\in\R;\,\,|f(x)|>\|f\|_\infty-\epsilon\}$. Since $f$ is continuous, we have $X_\epsilon\neq\emptyset$. Moreover, due to $f\in L^1(\R)$, we can find $L_\epsilon>0$ such that
$$
0<L_\epsilon(\|f\|_\infty-\epsilon)=\int_{X_\epsilon}(\|f\|_\infty-\epsilon)dx\leq \int_{X_\epsilon}|f(x)|dx\leq \|f\|_1,
$$
therefrom we obtain
$$
L_\epsilon(\|f\|_\infty-\epsilon)^p\leq \|f\|_p^p\leq \|f\|_r^{r/p}\|f\|_\infty^{1-r/p},
$$
which implies the result.
\end{proof}

Proposition \ref{prop3.3} has as a straightforward consequence the following fact: given a function $f$ in the class $L^1(\R)\cap L^\infty(\R)\cap C^0(\R)$, then $\|f\|_p\rightarrow\|f\|_\infty$, as $p\rightarrow\infty$, a fact that we shall use very often in section \ref{sec4}.

\begin{proposition}\label{prop3.4}
Suppose that $v$ is a continuous sub-multiplicative function such that $\inf_\R v>0$ and $e^{-|\cdot|}v(\cdot)\in L^1(\R)$. Then $e^{-|\cdot|}v(\cdot)\in L^p(\R)$, $2\leq p\leq\infty$.
\end{proposition}

\begin{proof}
Let us define $h(x)=e^{-|x|}v(x)$. In particular, $h$ is non-negative, continuous, and a member of $L^1(\R)$ by construction. We only need to show that $h\in L^\infty(\R)$, because then the result will be a foregone conclusion of proposition \ref{prop3.2}.

Let $I:=\{x\in\R;\,h(x)>1\}$. If $I=\emptyset$, then $h$ is clearly bounded. By assuming $I\neq\emptyset$, we can find a cover $(a_n,b_n)$, $n\in\N'\subseteq\N$ such that $h(a_n)=h(b_n)=1$ and
$$
I\subseteq\bigcup_{n\in\N'}(a_n,b_n).
$$
If the set $\N'$ is finite, say $\N'=\{1,\cdots,n\}$, then $I$ is contained in a compact set $J$ and then, $\|h\|_\infty=\max\limits_{J}|h(x)|$.

The case in which $\N'$ is infinite in more challenging. We first note that we can only have a finite number of indices for which $b_n-a_n\geq1$, otherwise the condition 
$$\int_\R|h(x)|dx<\infty$$
would not be true.

Let $n_0\in\N'$ such that $n>n_0$ implies $b_n-a_n<1$. Moreover, we can take $n_0$ large enough so that either $I_n\subseteq(0,\infty)$ or $I_n\subseteq(-\infty,0)$. Let $J$ be the closure of the reunion of the intervals with indexes up to $n_0$. Hence, $J$ is a compact set.

{\bf Claim:} We claim that for any $n>n_0$, $\sup\limits_{x\in I_n}h(x)\leq \sup\limits_{x\in[-1,1]}h(x)$.

If our claim is true, then we have
$$
\sup_{x\in\R}|h(x)|=\sup_{x\in I}|h(x)|=\sup_{x\in J\cup[-1,1]}|h(x)|,
$$
and the same argument used for the case $\N'$ assures that $h$ is bounded.

We now prove the claim, which is divided in two mutually exclusive cases. We recall that $h(a_n)=h(b_n)=1$.
\begin{itemize}
    \item {\bf Case $I_n\subseteq(-\infty,0)$.} For any $x,y\in (-\infty,0)$, note that 
    $$h(x+y)=e^{-|x+y|}v(x+y)\leq e^xv(x)e^y v(y)=h(x)h(y).$$
    For any $x\in I_n$, we can decompose it as $x=b_n+t$, for some $t\in(-1,0)$, and then
    $$h(x)=h(b_n+t)\leq h(b_n)h(t)\leq\sup_{t\in[-1,0]}h(t).$$
    
    \item {\bf Case $I_n\subseteq(-\infty,0)$.} Similarly as in the previous case, each $x\in I_n$ can be written as $x=a_n+t$, for some $t\in(0,1)$, and we also have $h(y+z)\leq h(y)h(z)$, for all $y,z\in(0,\infty)$. Thus, we have
    $$h(x)=h(a_n+t)\leq h(a_n)h(t)\leq\sup_{t\in[0,1]}h(t).$$
\end{itemize}
In both cases, we conclude that $h$ is bounded over the sets $I_n$ by $\sup\limits_{x\in[-1,1]}h(x)$.
\end{proof}

\begin{proposition}\label{prop3.5}
Assume that $\phi$ is a locally absolutely continuous function $\phi:\R\rightarrow\R$ such that $|\phi'(x)|\leq A |\phi(x)|$ a.e., for some $A>0$, and let us define, for each $N\in\N$, the function
\bb
\phi_N(x)=\min\{\phi(x),N\}.
\ee
Then,
\begin{enumerate}[(a)]
\item $(\phi_N)_{N\in\N}\subseteq C^0(\R)$ and is locally absolutely continuous;
\item $\phi_N(x)\leq\phi_{N+1}(x)\leq \phi(x)$, for each fixed $x\in\R$ and any $N\in \N$;
\item $(\phi_N)_{N}$ converges pointwise to $\phi(x)$;
\item $\|\phi_N\|_\infty\leq N$;
\item $|\phi'_N(x)|\leq A |\phi_N(x)|$ a.e.
\end{enumerate}
\end{proposition}

\begin{proof}
We begin by noticing that for any two functions $f$ and $g$, we have
$$\min\{f,g\}=\f{f+g}{2}-\f{|f-g|}{2}.$$
\begin{enumerate}[(a)]
\item We observe that both $\phi(\cdot)+N$ and $|\phi(\cdot)-N|$ are locally absolutely continuous. The definition of minimum given above and this observation prove the result;
\item Recall that $\phi_N(x)\leq \phi(x)$, for any $N\in\N$. Moreover, $\phi_N(x)=\phi_{N+1}(x)$, as long as $\phi(x)\leq N$, and for $x$ such that $\phi(x)>N$, then $\phi_{N+1}(x)> N\geq\phi_{N}(x)$;
\item Immediate;
\item Note that $\phi_N(x)\leq N$, for any $x\in\R$;
\item Consider the open sets $A_N:=\{x\in\R;\,\,\phi(x)<N$\}, $B_N:=\{x\in\R;\,\,\phi(x)>N$\}, and let $U_N=\R\setminus(A_N\cup B_N)$. If $A_N\neq\emptyset$ we have $\phi'_N(x)=\phi'(x)$, for any $x\in A_N$, whereas if $B_N\neq\emptyset$ and $x\in B_N$, then $\phi_N'(x)=0$. In any case, we have $|\phi_N'(x)|\leq A|\phi_N(x)|$. Finally, if $U_N\neq\emptyset$, and its interior is non-empty, then for any interior point $x$ we conclude that $\phi_N'(x)=0$, an already treated case. In case the interior of $U_N$ is empty, then we are forced to conclude that $U_N$ has measure $0$. In any case we conclude that $|\phi'_N(x)|\leq A |\phi_N(x)|$ a.e.
\end{enumerate}
\end{proof}

\begin{proposition}\label{prop3.6}
Let $\phi$ and $(\phi_N)_{N\in\N}$ as in Proposition \ref{prop3.5}. If $\phi$ is $v-$moderate, for some sub-multiplicative weight function $v$ with $\inf_\R v>0$, then the sequence $(\phi_N)_{N\in\N}$ is uniformly $v-$moderate with respect to $N$.
\end{proposition}

\begin{proof}
We only need to prove the existence of a constant $c_1>0$ that makes the inequality $\phi_N(x+y)\leq c_1v(x)\phi_N(y)$ true.

Since $\phi$ is $v-$moderate we can guarantee the existence of a positive constant $c_0$ such that \eqref{2.0.9} holds. For each $N\in\N$, define
$U_N=\{x\in\R;\,\,\phi(x)\leq N\}.$

If $y\in U_N$, we have 
$$\phi_N(x+y)\leq\phi(x+y)\leq c_0v(x)\phi_N(y).$$

On the other hand, since $\phi_N(z)\leq N$, for any $z\in\R$, taking $y\notin U_N$, then 
$$\phi_N(x+y)\leq N=\phi_N(y).$$
Let $\al:=\inf_R v$. Thus $\al^{-1}v\geq1$, $N\leq \al^{-1}v(x)N$, and $\phi_N(y)\leq \al^{-1}v(x)\phi_N(y)$.

Taking $c_1=\max\{c_0,\al^{-1}\}$, by \eqref{2.0.9} and the comments above we conclude that
$$
\ba{lcl}
\min\{\phi(x+y),N\}&\leq&\ds{\min\{c_1v(x)\phi(y),c_1v(x)N\}\leq c_1v(x)\min\{\phi(y),N\}}\\
\\
&=&\ds{c_1v(x)\phi_N(y)}.
\ea
$$
Therefore, $\phi_N(x+y)\leq c_1v(x)\phi_N(y)$.
\end{proof}

\section{Proof of the main results}\label{sec4}

\begin{theorem}\label{teo4.1}
Assume that ${\bf H}_1$--${\bf H}_3$ hold, $u\in C^0([0,T];H^s(\R))$, $s>3/2$, is a solution of \eqref{2.0.1}, and $(h_t(\cdot))_t$ and $(F_t(\cdot))_t$ are the families defined in \eqref{2.0.3} and \eqref{2.0.4}, respectively. Then
\begin{enumerate}[(a)]
\item $(h_t(\cdot))_{t\in[0,T]}\subseteq H^{s-1}(\R)$, and in particular, it is continuous;
\item $(F_t(\cdot))_{t\in[0,T]}\subseteq H^{s}(\R)$, and in particular, it is $C^1$;
\item $h_{t'}(x)=0$, for some $t'\in[0,T]$ and $x\in\R$, if and only if $u\equiv0$;

\item For each $(t,x)\in[0,T]\times\R$, define
$$
\sigma(t,x):=\left\{\ba{lcl}
\ds{\f{1}{t}\int_0^t h_{\tau}(x)d\tau},\quad t>0,\\
\\
\ds{h_0(x)}, \quad t=0.
\ea
\right.
$$
Then $\sigma(\cdot,\cdot)$ is a continuous function from $[0,T]\times\R$ to $\R$.
\end{enumerate}
\end{theorem}

\begin{proof}
\begin{enumerate}[(a)]
    \item The conditions on $u$ imply that $u_x\in H^{s-1}(\R)$, with $s>1/2$. By the algebra property \cite[page 320, exercice 6]{taylor}, we have $h_t(\cdot) \in H^{s-1}(\R)$ and the Sobolev Embedding Theorem (see \cite[page 317]{taylor} says that it is continuous;
    \item Consequence of the above result, the fact that the operator $\p_x\Lambda^{-2}$ applies $H^s(\R)$ into $H^{s+1}(\R)$, and the Sobolev Embedding Theorem (see \cite[page 317]{taylor};
    \item From ${\bf H}_2$ and ${\bf H}_3$ we conclude that $u(t',x)=0$, for any $x\in\R$, which gives $\|u(t',\cdot)\|_{H^1(\R)}=0$. By \eqref{2.0.10} we know that the Sobolev norm of the solution $u$ is time invariant, wherefrom we conclude that $u$ vanishes for each $t$ such that the solution exists.
    
    \item The continuity of $\sigma$ on $(0,T]\times\R$ comes from the continuity of $h_t(x)$. It is enough to prove its continuity when $t$ approaches $0$.
    
    The conditions on $u$ jointly with the well posedness result \cite[Theorem 3.1]{tian-nach} imply that $u$ is $C^1$ with respect to $t$, so that, taking $\epsilon>0$ sufficiently small and $0<t<\epsilon$, we have
    $$
    h_t(x)=h_0(x)+O(t),
    $$
    that is, $\sigma(t,x)=h_0(x)+O(t)$, implying the continuity of $\sigma$ near the line $\{0\}\times\R$.
\end{enumerate}
\end{proof}

Henceforth, $h_t(\cdot)$ and $F_t(\cdot)$ are the functions defined in \eqref{2.0.3} and \eqref{2.0.4}, while $M$ is a fixed constant satisfying \eqref{2.0.11}. They will be used throughout this section without further mention to their meaning or definition. 

Finally, let us emphasise a significant fact: some of our results are concerning for $p\in[2,\infty]$, see theorems \ref{teo2.4} and \ref{teo2.5}. In the demonstrations given in this section we replace $p$ by $2p$, and this implies that $p$ will be considered within $[1,\infty]$. The final outcome is exactly the same, but the demonstrations are significantly and technically simpler.

\subsection{Proof of theorem \ref{teo2.1} and its corollaries}\label{sec4.1}

{\bf Proof of Theorem \ref{teo2.1}.} Note that $F'_t(x)=\p_x F_t(x)=\Lambda^{-2}h_t(x)-h_t(x)$, where we used the identity $\p_x^2\Lambda^{-2}=\Lambda^{-2}-1$. As long as $x\in[a,b]$, then $F'_{t^\ast}(x)=\Lambda^{-2}h_{t^\ast}(x)$, and
$$
\int_{a}^b\Lambda^{-2}h_{t^\ast}(x)dx=\int_{a}^bF'_{t^\ast}(x)dx=F_{t^\ast}(b)-F_{t^\ast}(a)=0,
$$
implying that $h_{t^\ast}(x)=0$. The result is then a consequence of theorem \ref{teo4.1}. \hfill$\square$

{\bf Proof of Corollary \ref{cor2.1}.} By \eqref{2.0.1} we have 
\bb\label{4.1.1}
F_t(x)=-(u_t(t,x)+f'(u(t,x))u_x(t,x),
\ee
and from the conditions given we can find numbers $t^\ast$, $a$ and $b$ such that $\{t^\ast\}\times[a,b]\subseteq\Omega$ and $F_{t\ast}(a)=F_{t^\ast}(b)=0$. The result is an immediate consequence of theorem \ref{teo2.1}. \hfill$\square$

{\bf Proof of Corollary \ref{cor2.2}.} If the result were not true, we would then be able to find a set $\{t^\ast\}\times[a,b]$ such that $F_{t^\ast}(a)=F_{t^\ast}(b)$ (see \eqref{4.1.1}) and $f_{t^\ast}(x)=0$, $a<x<b$, which contradicts theorem \ref{teo2.1}. \hfill$\square$

{\bf Proof of Corollary \ref{cor2.3}.} If $u$ were compactly supported on $[0,T]\times\R$, we would be able to find a non-empty open set $\Omega$ such that $u\big|_\Omega\equiv0$. By corollary \ref{cor2.1} $u$ is trivial, which is a contradiction. \hfill$\square$

\subsection{Proof of theorem \ref{teo2.4}}\label{sec4.2}

\begin{theorem}\label{teo4.2}
Assume that ${\bf H_1}-{\bf H_3}$ hold, and let $u\in C^0([0,T];H^s(\R))$, $s>3/2$, be a solution of \eqref{2.0.1}. Suppose that $\phi$ and $\phi_N$ are a function and a sequence satisfying the conditions in proposition \ref{prop3.5}, respectively, and $g$ satisfy \eqref{2.0.5}. Then
\bb\label{4.2.1}
\f{d}{dt}\|\phi_N(\cdot)u(t,\cdot)\|_{2p}\leq M(A+M)\|\phi_N(\cdot)u(t,\cdot)\|_{2p}+CM\|\phi_N(\cdot) F_t(\cdot)\|_{2p}.
\ee
\end{theorem}

\begin{proof}
We may assume that $u$ is non-trivial, otherwise the result is immediate.

Multiplying \eqref{2.0.1} by $(\phi_Nu)^{2p-1}\phi_N$ and integrating the result with respect to $x$, we obtain
\bb\label{4.2.2}
\f{1}{2p}\f{d}{dt}\int_\R(\phi_Nu)^{2p}dx+I_1+I_2=0,
\ee
where
$$
I_1:=\int_\R((\phi_Nu)^{2p-1}\phi_N)^{2p-1}(\phi_Nf'(u)u_x)dx,
$$
and
$$
I_2:=\int_\R((\phi_Nu)^{2p-1}\phi_NF_t(x)dx.
$$

Integrating the identity
$$
\p_x((\phi_Nu)^{2p}f'(u))=2p((\phi_Nu)^{2p-1}(\phi'_Nu+\phi_Nu_x)f'(u)+((\phi_Nu)^{2p}f''(u)u_x
$$
over $\R$, taking into account that the conditions on $f$ imply that $u,\,f(u)\rightarrow 0$, as $|x|\rightarrow\infty$; $|\phi_N'|\leq A|\phi_N|$, we arrive at the estimate
\bb\label{4.2.3}
\ba{lcl}
|I_1|&\leq&\ds{\Big|\int_\R(\phi_Nu)^{2p-1}\phi_N'uf'(u)dx\Big|+\f{1}{2p}\Big|\int_\R(\phi_Nu)^{2p}f''(u)u_xdx\Big|}\\
\\
&\leq&\ds{(A+M)M\|\phi_Nu\|_{2p}^{2p}}.
\ea
\ee

Let us now estimate $I_2$: since $u$ is a solution of \eqref{2.0.1}, the Sobolev Embedding Theorem and the invariance of the functional \eqref{2.0.10} imply $\|u(t,\cdot)\|_\infty\leq\|u(t,\cdot)\|_{H^1}=\|u_0\|_{H^1}$. Thus,
$$
|I_2|\leq cM\int_\R|\phi_Nu|^{2p-1}|\phi_NF_t|dx.
$$

From \eqref{2.0.5} we conclude that $g(u(t,\cdot))\in L^2(\R)$, since
$$
\int_\R g(u(t,x))dx\leq c\int_\R u(t,x)^2dx\leq c \|u(t,\cdot)\|_{H^1(\R)}^2.
$$

This last observation, jointly with the conditions on $u$ and $\phi$, tell us that both $\phi_Nu$ and $\phi_NF_t$ belong to $L^1(\R)\cap L^\infty(\R)$. Applying proposition \ref{prop3.2} we conclude that $\phi_Nu\in L^{\f{2p}{2p-1}}(\R)$ and $\phi_NF_t\in L^{2p}(\R)$, which we combine with the Hölder inequality to conclude that
\bb\label{4.2.4}
I_2\leq cM\|\phi_Nu\|_{2p}^{2p-1}\|\phi_NF_t\|_{2p}.
\ee

From \eqref{4.2.3}, \eqref{4.2.4}, and \eqref{4.2.2} we obtain
$$
\|\phi_Nu\|_{2p}^{2p-1}\f{d}{dt}\|\phi_Nu\|_{2p}\leq (A+M)M\|\phi_Nu\|_{2p}^{2p}+cM\|\phi_Nu\|_{2p}^{2p-1}\|\phi_N F_t\|_{2p},
$$
which implies \eqref{4.2.1}.
\end{proof}

\begin{theorem}\label{teo4.3}
If $u\in C^0([0,T];H^s(\R))$, $s>3/2$, is a solution of \eqref{2.0.1}, and $\phi$ is a function satisfying proposition \ref{prop3.6}, then
\bb\label{4.2.5}
\f{d}{dt}\|\phi_N(\cdot) u_x(t,\cdot)\|_{2p}\leq M^2(A+1)\|\phi_N(\cdot)\,u_x(t,\cdot)\|_{2p}+\|\phi_N(\cdot) \p_x F_t (\cdot)\|_{2p}.
\ee
\end{theorem}

\begin{proof}
Let $p\geq1$. Differentiating \eqref{2.0.1} with respect to $x$, multiplying the result by $\phi_N(\phi_Nu_x)^{2p-1}$, and integrating over $\R$, we obtain,
\bb\label{4.2.6}
\f{1}{2p}\f{d}{dt}\int_\R(\phi_Nu_x)^{2p}dx+J_1+J_2+J_3=0,
\ee
where
$$
J_1:=\int_\R(\phi_Nu_x)^{2p}f''(u)u_xdx,\quad J_2:=\int_\R(\phi_Nu_x)^{2p-1}f'(u)\phi_Nu_{xx}dx,
$$
and
$$
J_3:=\int_\R(\phi_Nu_x)^{2p-1}\phi_N\p_xF_t(x)dx.
$$

Using the identity
$$\p_x(f'(u)(\phi_Nu_x)^{2p})=f''(u)(\phi_Nu_x)^{2p}u_x+2pf''(u)(\phi_Nu_x)^{2p-1}(\phi_N'u_x+\phi_Nu_{xx}),$$
the fact that $|f'(u)|,|f''(u)|\leq M$, and $|\phi'_N|\leq A|\phi_N|$, we conclude that
\bb\label{4.2.7}
|J_2|\leq M(A+M)\|\phi_Nu_x\|_{2p}^{2p}.
\ee

In a more straightforward way, we also obtain
\bb\label{4.2.8}
|J_1|\leq M^2\|\phi_Nu_x\|_{2p}^{2p}.
\ee

To estimate $J_3$ we note that $u_x\in L^2(\R)\cap L^\infty(\R)$, which jointly with $g(u(t,\cdot))\in L^2(\R)$ say that $\p_x F_t=\Lambda^{-2}h_t-h_t\in L^1(\R)\cap L^\infty(\R)$. Therefore,
\bb\label{4.2.9}
|J_3|\leq\|\phi_Nu_x\|_{2p}^{2p-1}\|\phi_N\p_x F_t\|_{2p}.
\ee

From \eqref{4.2.7}--\eqref{4.2.9} and \eqref{4.2.6} we obtain
$$
\|\phi_Nu_x\|_{2p}^{2p-1}\f{d}{dt}\|\phi_Nu_x\|_{2p}\leq M(A+2M)\|\phi_Nu_x\|_{2p}^{2p}+\|\phi_Nu_x\|_{2p}^{2p-1}\|\phi_N\p_x F_t\|_{2p}.
$$

If $u\equiv0$, then \eqref{4.2.5} is trivially satisfied, otherwise, it is straightforwardly implied by the inequality above.
\end{proof}

\begin{theorem}\label{teo4.4}
Assume that ${\bf H_1}-{\bf H_3}$ hold, $\phi$ is an admissible weight function for \eqref{2.0.1}, and $u\in C^0([0,T];H^s(\R))$, $s>3/2$, is a solution of \eqref{2.0.1}. Then there exists a constant $c>0$, independent of $t$, such that 
$$\max\{\|\phi_N(\cdot)F_t(\cdot)\|_{2p},\|\phi_N(\cdot) \p_x F_t(\cdot)\|_{2p}\}<c\|\phi_N(\cdot) h_t(\cdot)\|_{2p}.$$
\end{theorem}

\begin{proof} Given that $\phi$ is an admissible weight function for \eqref{2.0.1} (see definition \ref{def2.1}), we conclude the existence of $A>0$ and a continuous function $v$, with $\inf_Rv>0$, such that $e^{-|\cdot|}v\in L^{1}(\R)$ and $|\phi'(x)|\leq A |\phi(x)|$. On the other hand, $F_t(x)=(\p_xp)\ast h_t$, where $p$ is given by \eqref{2.0.2}, and $\p_x F_t=(p\ast h_t)-h_t$. By Lemma \ref{lema3.2}, we have
\bb\label{4.2.10}
\ba{lcl}
\|\phi_NF_t\|_{2p}&=&\|\phi_N((\p_xp)\ast h_t)\|_{2p}\leq c_0\|(\p_xp) v\|_1\|\phi_N h_t\|_{2p},\\
\\
\|\phi_N\p_x F_t\|_{2p}&\leq&\|\phi_N(p\ast h_t)\|_{2p}+\|\phi_Nh_t\|_{2p}\leq \|p\phi_N\|_1\|\phi_N h_t\|_{2p}+\|\phi_Nh_t\|_{2p},
\ea
\ee
noticing that $e^{-|\cdot|}v\in L^1(\R)$, then $\|(\p_xp) v\|_1\leq\|pv\|_1\leq\|e^{-|\cdot|}v\|_1=:a$, for some positive constant $a$. On the other hand,
$$
\ba{lcl}
\|\phi_N \p_x F_t\|_{2p}&\leq&\ds{ \|\phi_N(p\ast h_t)\|_{2p}+\|\phi_Nh_t\|_{2p}\leq (c_0\|pv\|_1+1)\|\phi_N h_t\|_{2p}}\\
\\
&\leq&\ds{ (ac_0+1)\|\phi_N h_t\|_{2p}},
\ea
$$
for some $c_0>0$. The result follows by taking $c=ac_0+1$.

\end{proof}

\begin{theorem}\label{teo4.5}
Under the conditions in theorem \ref{teo4.4}, there exits a constant $c>0$, independent of $t$, such that
$$\|\phi_N(\cdot) h_t(\cdot)\|_{2p}\leq c\Big(\|\phi_N(\cdot) u(t,\cdot)\|_{2p}+\|\phi_N(\cdot) u_x(t,\cdot)\|_{2p}\Big).$$
\end{theorem}

\begin{proof} The conditions on $g$ and $u$ guarantee the existence of a constant $k>0$ such that $|g(u)|\leq k |u|$. On the other hand, we have
$$
|\phi_Nh_t|=\Big|\phi_N\Big(g(u)+\f{f''(u)}{2}u_x^2\Big)\Big|\leq \phi_N\Big(k|u|+\f{M^2}{2}|u_x|\Big),
$$
where we used that $|f''(u)|,\,|u_x|\leq M$. Therefore, for some $\kappa>0$, we have
$$
\ba{lcl}
|\phi_Nh_t|^{2p}&\leq&\ds{\kappa^{2p}|\phi_N|^{2p}(|u|+|u_x|)^{2p}\leq \kappa^{2p}|\phi_N|^{2p}(2\max\{|u|,|u_x|\})^{2p}}\\
\\
&\leq &\ds{K|\phi_N|^{2p}(|u|^{2p}+|u_x|^{2p})},
\ea
$$
for some $K>0$. Thus, 
$$
\ba{lcl}\|\phi_Nh_t\|_{2p}^{2p}&\leq& K(\|\phi_Nu\|_{2p}^{2p}+\|\phi_Nu_x\|_{2p}^{2p})\leq 2K(\max\{\|\phi_Nu\|_{2p},\|\phi_Nu_x\|_{2p}\})^{2p}\\
\\
&\leq& 2K\Big(\|\phi_Nu\|_{2p}+\|\phi_Nu_x\|_{2p}\Big)^{2p},
\ea
$$
which proves the result.
\end{proof}

{\bf Proof of theorem \ref{teo2.4}.} Let $\eta_N(t):=\|\phi_N(\cdot)u(t,\cdot)\|_{2p}+\|\phi_N(\cdot)u_x(t,\cdot)\|_{2p}$. By theorems \ref{teo4.2}--\ref{teo4.4}, we have
\bb\label{4.2.11}
\f{d}{dt}\eta_N\leq k\eta_N,
\ee
for some constant $k>0$ depending on $M$, $A$, and $c$. Therefore, for a suitable constant $\kappa$ depending on both $k$ and $T$, the Grönwall inequality implies
$$
\|\phi_N(\cdot)u(t,\cdot)\|_{2p}+\|\phi_N(\cdot)u_x(t,\cdot)\|_{2p}\leq\kappa\Big(\|\phi_N(\cdot)u_0(\cdot)\|_{2p}+\|\phi_N(\cdot)u_0'(\cdot)\|_{2p}\Big).
$$

Since the constants do not depend on $N$, we can take $N\rightarrow\infty$ and obtain 
\bb\label{4.2.12}
\|\phi(\cdot)u(t,\cdot)\|_{2p}+\|\phi(\cdot)u_x(t,\cdot)\|_{2p}\leq\kappa\Big(\|\phi(\cdot)u_0(\cdot)\|_{2p}+\|\phi(\cdot)u_0'(\cdot)\|_{2p}\Big).
\ee

The inequality \eqref{4.2.12} holds for any $1\leq p<\infty$ and the constant $\kappa$ does not depend on it, so that we can take the limit $p\rightarrow\infty$ to conclude its validity for $1\leq p\leq\infty$. \hfill$\square$

\subsection{Proof of theorem \ref{teo2.5}}\label{sec4.3}

\begin{theorem}\label{teo4.6}
Let $\phi:\R\rightarrow\R$ be a locally absolutely continuous function, such that $|\phi'(x)|\leq A |\phi(x)|$ a.e., for some $A>0$, $v-$moderate for some continuous, sub-multiplicative weight function $v$ satisfying both $\inf_\R v>0$ and $e^{-|\cdot|}v(\cdot)\in L^p(\R)$, $2\leq p\leq \infty$. If $u_0\in H^s(\R)$, $s>3/2$, satisfies \eqref{2.0.13}, then the corresponding solution $u=u(t,x)$ of \eqref{2.0.1}, with initial condition $u(0,x)=u_0(x)$, inherits the same property, more precisely,
    $$\sup_{t\in[0,T]}\Big(\|\phi(\cdot)^{\f{1}{2}}u(t,\cdot)\|_2+\|\phi(\cdot)^{\f{1}{2}}u_x(t,\cdot)\|_2\Big)<\infty.$$
\end{theorem}

We would like to note that any admissible weight function for the equation \eqref{2.0.1} satisfies the conditions in the theorem \ref{teo4.6} in view proposition \ref{prop3.4}, see also \cite[page 5172]{bran-imrn}, meaning that the result above can be applied to any function in the sense of definition \ref{def2.1}.

In addition, it is worth noticing that if, instead of $f\in L^1(\R)$ in proposition \ref{prop3.3}, we assume $f\in L^q(\R)$, for some $q\in(1,\infty)$, (and maintain all the remaining conditions) then the conclusion of that proposition would still  be true for some $r$ and $p$ sufficiently large satisfying $q\leq r<p$. While in that result the constant $L_\epsilon$ was controlled by $\|f\|_1$, with this new hypothesis it is now controlled by $\|f\|_q$. As a result, we again have $\|f\|_p\rightarrow\|f\|_\infty$ as $p\rightarrow\infty$.

\begin{proof}
We begin with by noticing that $|(\phi^{1/2})'|=(1/2)|\phi^{-1/2}\phi'|\leq (A/2)|\phi^{1/2}|$, that is, $\phi^{1/2}$ is a $v^{1/2}-$moderate weight. Therefrom $e^{-|\cdot|}v^{1/2}\in L^{2p}(\R)$ and whenever 
$$\f{1}{2p}+\f{1}{q}=1$$
is satisfied, the Hölder inequality tell us that
$$\|e^{-|\cdot|}v^\f{1}{2}\|_{1}\leq\|e^{-\f{|\cdot|}{2}}v^\f{1}{2}\|_{2p}\|e^{-\f{|\cdot|}{2}}\|_q,$$
wherefrom we conclude that $e^{-|\cdot|}v^\f{1}{2}\in L^1(\R)$ and the result follows from theorem \ref{teo2.4} with $p=2$ and $\phi$ replaced by $\phi^{\f{1}{2}}$.
\end{proof}

\begin{theorem}\label{teo4.7}
Suppose ${\bf H_1}-{\bf H_3}$ hold, $u$, $u_0$, and $\phi$ satisfy the conditions in theorem \ref{teo4.6}, and $g(u(t,x))\leq c u(t,x)^2$, for some constant $c>0$. Then there exists a constant $c'>0$, depending only on $\phi$ and the initial datum, such that 
$$
\|\phi_N(\cdot)h_t(\cdot)\|_1<c',\,\,\,\,\|\phi_N(\cdot) F_t(\cdot)\|_p<c',
$$
and
$$|\phi_N(\cdot)\p_xF_t(\cdot)\|_p\leq c' +M(\|\phi_N(\cdot)u(t,\cdot)\|_p+\|\phi_N(\cdot)u_x(t,\cdot)\|_p).
$$
\end{theorem}

\begin{proof}
Noticing that $\phi_N\leq\phi$, we have
\bb\label{4.3.1}
\ba{lcl}
\|\phi_Nh_t\|_1 & \leq & \ds{\int_\R\phi_N\Big(|g(u)|+\f{|f''(u)|}{2}u_x^2\Big) dx}\\
\\
&\leq&\ds{\max\{c,\f{M}{2}\}\int_\R\Big((\phi^\f{1}{2}u)^2+(\phi^\f{1}{2}u_x)^2\Big)dx}\\
\\
&= & \ds{\max\{c,\f{M}{2}\}\Big(\|\phi^{\f{1}{2}}u(t,\cdot)\|_2^2+\|\phi^{\f{1}{2}}u_x(t,\cdot)\|_2^2\Big).}
\ea
\ee

By Lemma \ref{lema3.2} and taking into account that $\p_xp(\cdot)=\sign{(\cdot)}e^{-|\cdot|}/2$, we have
\bb\label{4.3.2}
\|\phi_NF_t\|_p=\|((\p_xp)\ast h_t)\phi_N\|_p\leq c_0\|\phi_N h_t\|_1\|(\p_xp)v\|_p\leq \f{c_0\|e^{-|\cdot|}v\|_p}{2}\|\phi_N h_t\|_1.
\ee

Finally, noticing that $\p_x F_t=p\ast h_t-h_t$, and applying once again lemma \ref{lema3.2}, we have
\bb\label{4.3.3}
\|\phi_N\p_xF_t\|_p\leq\f{1}{2}\|e^{-|\cdot|}v\|_p\|\phi_Nh_t\|_1+\|\phi_Nh_t\|_p.
\ee

By theorem \ref{teo4.6} we see that the right hand side of \eqref{4.3.1} can be bounded by a constant depending on the initial datum, $T>0$, and $\phi$. The same argument also applies to \eqref{4.3.2}, once the fact that $e^{-|\cdot|}v\in L^p(\R)$ is taken into account, we conclude that \eqref{4.3.2} is also bounded. Therefore, we can find a constant $\sigma$ fulfilling the thesis in the theorem, whereas its remaining part is a consequence of theorem \ref{teo4.5}.
\end{proof}

{\bf Proof of theorem \ref{teo2.5}.} Inequality \eqref{2.0.15} is an immediate consequence of theorem \ref{teo4.6}.

Let us now define 
$$\eta_N(t):=\|\phi_N(\cdot)u(t,\cdot)\|_{2p}+\|\phi_N(\cdot)u_x(t,\cdot)\|_{2p}.$$

Proceeding similarly as in the proof of theorem \ref{teo2.4}, we have
$$
\f{d}{dt}\eta_N(t)\leq k\Big(\eta_N(t)+\|\phi_N(\cdot)F_t(\cdot)\|_{2p}+\|\phi_N(\cdot)\p_x F_t(\cdot)\|_{2p}\Big)\leq k(\eta_N(t)+c),
$$
where we used theorems \ref{teo4.2}--\ref{teo4.4} and \ref{teo4.7} to conclude the existence of the constants $k$ and $c$, independent on $p$, $t$ and $N$.

In view of Grönwall's inequality, we obtain
$$
\eta_N(t)\leq \eta_N(0)e^{kT}+\f{ck}{e^{kT}-1}.
$$

Given that the constants are independent of $N$, we can make $N\rightarrow \infty$ and then arrive at
$$
\|\phi(\cdot)u(t,\cdot)\|_{2p}+\|\phi(\cdot)u_x(t,\cdot)\|_{2p}\leq \Big(\|\phi(\cdot)u_0(\cdot)\|_{2p}+\|\phi(\cdot)u_0'(\cdot)\|_{2p}\Big)e^{kT}+\f{c}{k}(e^{kT}-1),
$$
which proves the result for $p\in[1,\infty)$. Since the constants involved are independent of $p$ we can consider $p\rightarrow\infty$ and conclude the demonstration.\hfill$\square$

\subsection{Proof of theorem \ref{teo2.3} and \ref{teo2.2}}\label{sec4.4}

Let $u\in C^0([0,T];H^s(\R))$, $s>3/2$ be a solution of \eqref{2.0.1}, $h_t(\cdot)$ the function given in \eqref{2.0.3}, and let us consider the quantity
$$
\al(t):=\int_\R e^{|x|}h_t(x)dx.
$$

If \eqref{2.0.5} is satisfied, then
$$
0\leq \al(t) \leq \max\{c,\f{M}{2}\}\int_\R\Big((e^{\f{|x|}{2}}u(t,x))^2+(e^{\f{|x|}{2}}u_x(t,x))^2\Big)dx
$$

By theorem \ref{teo4.6}, with $\phi(x)=e^{|x|}$, we conclude that the last integral above is well defined. Therefore, the quantity $\al(\cdot)$ above is a well defined continuous function and, as such, it is uniformly bounded on the compact set $[0,T]$. 

For each $t\in[0,T]$ we define
\bb\label{4.4.1}
\lambda_\pm(t)=\f{1}{2}\int_\R e^{\pm y}\sigma(t,y)dy,
\ee
where $\sigma$ is the function given in theorem \ref{teo4.1}. 

A simple calculation shows that
$$
0\leq \lambda_\pm(t)\leq\f{1}{2t}\int_0^t\al(\tau)d\tau,\,\,t>0,
$$
and proceeding similarly as in the proof of theorem \ref{teo4.1}, we also have $0\leq\lambda_\pm(0)\leq\al(0)/2$.

Therefore, $\lambda_\pm(\cdot)$ is uniformly bounded on the compact set $[0,T]$, that is, there are some non-negative constants $c_1$ and $c_2$ such that
\bb\label{4.4.2}
0\leq c_1\leq\lambda_\pm(t)\leq c_2.
\ee

Moreover,
\bb\label{4.4.3}
\ba{lcl}
\ds{-\int_0^t F_\tau(x)d\tau}&=&\ds{\f{1}{2}\int_0^t\int_\R\sign{(x-y)}e^{-|x-y|}h_\tau(y)dyd\tau}\\
\\
&=&\ds{\f{1}{2}\int_\R\sign{(x-y)}e^{-|x-y|}\Big(\underbrace{\int_0^th_\tau(y)dy}_{t\sigma(t,y)}d\tau\Big)dy dy}\\
\\
&=&\ds{\f{t}{2}e^{-x}\int_{-\infty}^xe^y\sigma(t,y)dy-\f{t}{2}e^{x}\int^{\infty}_xe^{-y}\sigma(t,y)dy}.
\ea
\ee

For $x\gg 1$, we have
\bb\label{4.4.4}
\ba{lcl}
\ds{-\int_0^t F_\tau(x)d\tau}&=&\ds{te^{-x}\Big[\f{1}{2}\int_{-\infty}^x e^y\sigma(t,y)dy-\f{e^{2x}}{2}\int_{-\infty}^x e^{-y}\sigma(t,y)dy\Big]}\\
\\
&=&\ds{t e^{-x}\Big[\f{1}{2}\Big(\int_{-\infty}^x+\int_x^\infty\Big)e^y\sigma(t,y)dy\Big]}\\
\\
&&\ds{-t e^{-x}\Big[\f{1}{2}\int_x^\infty e^y\sigma(t,y)dy+\f{e^{2x}}{2}\int_x^\infty e^{-y}\sigma(t,y)dy\Big]}\\
\\
&=&\ds{t e^{-x}(\lambda_+(t)+\epsilon_+(t,x)).
}
\ea
\ee

A similar procedure shows, for $x\ll1$, that
\bb\label{4.4.5}
-\int_0^t F_\tau(x)d\tau=t e^{x}(-\lambda_-(t)-\epsilon_-(t,x)),
\ee
where $\lambda_\pm$ is given by \eqref{4.4.1} and
\bb\label{4.4.6}
\varepsilon_\pm(t,x)=\f{1}{2}\int^x_{\pm\infty} (e^{\pm y}+e^{\pm(2x-y)})\sigma(t,y)dy.
\ee

{\bf Proof of theorem \ref{teo2.3}.} Integrating \eqref{2.0.1} and using \eqref{4.4.1}, \eqref{4.4.3}--\eqref{4.4.6}, we have
\bb\label{4.4.7}
u(t,x)=u_0(x)-\int_0^t f'(u(\tau,x))u_x(\tau,x)d\tau \pm te^{-| x|}(\lambda_\pm(t)+\varepsilon_\pm(t,x)),
\ee
where $+$ and $-$ are taken in agreement with the sign of $x$.

Since $u$ is bounded, $f\in C^\infty(\R)$, and $f'(0)=0$, we conclude that $f'$ is locally Lipschitz, meaning that $|f'(u)|\leq c|u(t,x)|$, for some constant $c>0$ depending on $\|u(0,x)\|_{H^1(\R)}$. Since $u$ vanishes as $|x|\rightarrow\infty$, then $|f'(u)|\sim O(u)$ as $|x|\rightarrow\infty$, then \eqref{2.0.16}, with $p=\infty$, implies that $f'(u)u_x\sim O(e^{-|x|}(1+|x|)^{-1}(\ln{(1+|x|))^{-2d}})$. Therefore, defining
$${\cal R}(t,x)=-\int_0^t f'(u(\tau,x))u_x(\tau,x)d\tau$$
we see that 
$$\f{e^{x}|{\cal R}(t,x)|}{t}\leq L,$$
for some constant $L>0$, wherefrom we are forced to conclude that (eventually after incorporating possible time dependent terms in $\varepsilon_\pm)$
\bb\label{4.4.8}
{\cal R}(t,x)=t e^{-|x|} R(x),
\ee
for some function $R$ satisfying 
\bb\label{4.4.9}
R(x)\sim O((1+|x|)^{-1}(\ln{(1+|x|))^{-2d}}).
\ee

Taking $\Phi_\pm(\cdot)=\lambda_\pm(\cdot)$ and $\epsilon_\pm=\varepsilon_\pm$ we now only need to show that $\epsilon_\pm(t,x)\rightarrow0$ as $\pm x\rightarrow\infty$.

Let $\varepsilon_+$ given by \eqref{4.4.6}. Then, we have
$$
0\leq \epsilon_+(t,x)=\f{1}{2}\int^{\infty}_xe^{y}(1+e^{2(x-y)})\sigma(t,y)dy\leq\int^{\infty}_x e^{y}\sigma(t,y)dy\rightarrow0,\quad x\rightarrow\infty.
$$

Similarly, we have
$$
0\leq \epsilon_-(t,x)=\f{1}{2}\int_{-\infty}^xe^{-y}(1+e^{2(y-x)})\sigma(t,y)dy\leq\int_{-\infty}^x e^{-y}\sigma(t,y)dy\rightarrow0,\quad x\rightarrow -\infty.
$$

Inequality \eqref{4.4.2} concludes the demonstration. \hfill$\square$

{\bf Proof of theorem \ref{teo2.2}.}

In view of the time translation invariance of the equation, we may assume that $u$ is compactly supported at $t=0$. Our aim is to show that this situation cannot occur for any other further time.

The fact that $u_0$ is compactly supported, tells us that it satisfies \eqref{2.0.6}. Moreover, the conditions \ref{teo2.3} are fulfilled and then \eqref{2.0.7} holds.

Assume that $u_0$ is non-trivial and $u$ is compactly supported at $t=t_1>0$. From \eqref{2.0.8} we have
$$
u(t_1,x)=u_0(x)\pm\, t_1\,e^{-|x|}\Big(\Phi_\pm(t_1)+\epsilon_\pm(t_1,x)+R(x)\Big).
$$

Therefore, taking into account that 
\bb\label{4.4.10}
\Phi_+(t_1)\gg\epsilon_+(t_1,x)+R(x),
\ee
we have
$$
0=\Phi_+(t_1)=\f{1}{2}\int_0^{t_1}e^{y}\sigma(t_1,y)dy,
$$
as long as we take $x$ large enough, that is, $x\gg x_0:=|\max\{\supp{(u_0(\cdot))}\cup\supp{(u(t_1,\cdot))}\}|$.

By construction, $\sigma(t,x)$ is non-negative (see its definition in theorem \ref{teo4.1}) and it vanishes, for $t>0$, if and only if $h_t(x)=0$, $x\in\R$. Again, by theorem \ref{teo4.1}, we conclude that $u\equiv0$, which contradicts the non-triviality of $u$.
\hfill$\square$

%\subsection{An improvement result}\label{subsec4.5}

%Let us take a look at \eqref{2.0.1}. Note that if $u=u(t,x)$ is a solution of that equation, than $v(t,x)=u(t,-x)$ is a solution of
%\bb\label{4.5.1}
%u_t-f'(u)u_x-\p_x\Lambda^{-2}\Big(g(u)+\f{f''(u)}{2}u_x^2\Big)=0.
%\ee

%Moreover, a careful analysis of our demonstrations shows that what is really crucial is that the non-local term is non-negative. Moreover, equation \eqref{4.5.1} can be recovered from \eqref{2.0.1} by replacing $f\mapsto-f$ and $g\mapsto-g$. Therefore, our results would still be true if instead of $f''(x)\geq0$ and $g(x)\geq 0$ in ${\cal H}_1$ and ${\cal H}_2$, we assume $f''(x)\leq0$ and $g(x)\leq 0$. In this case, for consistence, condition \eqref{2.0.5} should be replaced by $|g(u(t,x))|\leq cu(t,x)^2$ (which would again imply the original version again).

\section{Examples and applications}\label{sec5}

Here we apply our results to some concrete models belonging to the class \eqref{1.0.7}.

\subsection{The BBM equation}

Let us consider the BBM equation \cite{bbm}
\bb\label{5.1.1}
u_t-u_{txx}+uu_x=0,
\ee
that can be obtained from \eqref{1.0.7} by taking $f(u)=0$ and $g(u)=u^2/2$. Clearly condition ${\bf H_1}-{\bf H_3}$ are satisfied. As a consequence, if $u\in C^0([0,T];H^s(\R))$ vanishes on an open set, then it necessarily vanishes everywhere, in view of corollary \eqref{cor2.2}. This observation, in fact, is a straightforward consequence of the results reported in \cite{raspa-mo} for the more general class of BBM-type equation
\bb\label{5.1.2}
u_t-u_{txx}+\p_x g(u)=0,
\ee
for a non-negative smooth function $g$ vanishing only at $0$. However, in \cite{raspa-mo} the problem of asymptotic profile of such equation was not considered. If we suppose that $g$ satisfies the condition $g(u)\leq cu^2$, then the corresponding solution of \eqref{5.1.2} emanating from an initial datum satisfying \eqref{2.0.6} have the form \eqref{2.0.8}, which means that it cannot be spatially compactly supported for two different times. In particular, this holds for \eqref{5.1.1}.

\subsection{Dai's equation}

Let us now consider Dai's equation \eqref{1.0.3}. Taking $f(u)=\gamma u^2/2$ and $g(u)=(3-\gamma)u^2/2$ we obtain \eqref{1.0.3} from \eqref{1.0.7} (note that the CH equation corresponds to the case $\gamma=1$). 

Again our conditions on $f$ and $g$ are satisfied, including \eqref{2.0.5}, as long as $\gamma\in[0,3]$. Therefore, by theorem \ref{teo2.1} we conclude that a solution $u\in C^0([0,T];H^s(\R))$, $s>3/2$, emanating from a non-trivial initial data cannot vanish on any open set of its domain, nor can be compactly supported in two different times, in view of theorem \ref{teo2.3}. Moreover, solutions arising from initial data with strong decay (in the sense of \eqref{2.0.6}) will also inherit such property in view of theorem \ref{teo2.3}.

\subsection{The DGH equation}

Let us now consider the DGH equation \cite{dgh} equation
\bb\label{5.3.1}
 u_t-u_{txx}+3uu_x=2u_xu_{xx}+uu_{xxx}+\al u_x+\Gamma u_{xxx}.
\ee

The presence of the linear advection $\al u_x$ may cause a problem at first sight, once it would violate condition \eqref{2.0.5}. However, such a term can be eliminated under the change $u(t,x)=v(t,x+\al t)$. Note that Sobolev spaces are invariant under translations (in the sense that if $f(\cdot)\in H^s(\R)$, then $f(\cdot+x)\in H^s(\R)$). Moreover, the map $(t,x)\mapsto (t,x+\al t)$ is a global diffeomorphism fixing the line $\{0\}\times\R$, meaning that it preserves initial datum, so that $u\in C^0([0,T];H^s(\R))$ is a solution of \eqref{5.3.1} subject to $u(t,x)=u_0(x)$ if and only if $v$ is a solution of the equation
\bb\label{5.3.2}
 v_t-v_{txx}+3vv_x=2v_xv_{xx}+vv_{xxx}+\hat{\Gamma} v_{xxx}
\ee
subject to the same initial condition. Above, $\hat{\Gamma}=\Gamma-\al$. Equation \eqref{5.3.2} is a particular case of \eqref{1.0.7} with $g(u)=u^2$ and $f(u)=u^2+\hat{\Gamma}u$. 

Clearly the functions $f$ and $g$ above satisfy ${\bf H_1}-{\bf H_3}$, and \eqref{2.0.5} as well. Therefore, by theorem \ref{teo2.1} we conclude that the only solution $u\in C^0([0,T];H^s(\R))$ of the DGH equation vanishing on an open set $\Omega\subseteq[0,T]\times\R$ is $u\equiv0$. However, theorems \ref{teo2.2} and \ref{teo2.3} can only be applied whenever $\hat{\Gamma}=0$. If this is so, then \eqref{5.3.2} becomes the CH equation and we re-obtain the results in \cite{bran-imrn}, see also \cite{freire-jpa,freire-cor}.

\subsection{The rCH equation}

Let us now consider equation \ref{1.0.5}. In view of our discussions about the DGH equation, without loss of generality we may assume the form
\bb\label{5.4.1}
 u_t-u_{txx}+3uu_x=2u_xu_{xx}+uu_{xxx}+\be u^2u_x+\gamma u^3u_x+\Gamma u_{xxx},
\ee
that can be obtained from \eqref{1.0.5} by taking
$f(u)=u^2/2+\Gamma u$ and
\bb\label{5.4.2}
g(u)=\Big(1+\f{\be}{3}u+\f{\gamma}{4}u^2\Big)u^2
\ee

Conditions ${\bf H_1}$ and ${\bf H_2}$ are trivially satisfied, but ${\bf H_3}$ is more delicate, since \eqref{5.4.2} is not necessarily non-negative. However, in case $\be^2<9\gamma$, we then have
$$
1+\f{\be}{3}u+\f{\gamma}{4}u^2>\f{9\gamma-\be^2}{9\gamma}>0,
$$
where above we are assuming $\gamma>0$. Moreover, if $u\in C^0([0,T];H^s(\R))$ is a solution of \eqref{5.4.1}, then it is unique since it arises from the initial datum $u_0(x):=u(0,x)$ in view of \cite[Theorem 1.1]{raspajde}. Therefore, in view of Sobolev Embedding Theorem, we have $\|u(t,\cdot)\|_\infty\leq\|u_0\|_{H^1(\R)}$, that, jointly \eqref{5.4.2}, enable us to conclude that
$$
g(u)\leq \Big(1+\f{|\be|}{3}\|u_0\|_{H^1(\R)}+\f{\gamma}{4}\|u_0\|_{H^1(\R)}^2\Big)u^2=:cu^2,
$$
and then, \eqref{2.0.5} is fulfilled.

Likewise the DGH equation, theorems \ref{teo2.2} and \ref{teo2.3} can only be applied if $\Gamma=0$. Therefore, as long as $\Gamma=0$, $\gamma>0$, and $\be^2<9\gamma$, any solution $u\in C^0([0,T];H^s(\R))$, $s>3/2$, of \eqref{5.4.1} vanishing on a non-empty open set $\Omega\subseteq[0,T]\times\R$ must vanish everywhere. Also, if the constraints on the parameters hold and $u$ is a solution of \eqref{5.4.1} with initial condition satisfying \eqref{2.0.6}, then $u$ has the asymptotic profile \eqref{2.0.8}, which prevents the existence of compactly supported solutions for multiple values of time.

\section{Discussion}\label{sec6}

Since Constantin's paper \cite{const-jmp} showed that any solution of the Camassa-Holm equation emanating from a compactly initial data loses such property instantly, many approaches or ways for tackling the same problem have been introduced, see \cite{henry-jnmp,him-cmp, bran-imrn}. All of these works have as a consequence the establishment of persistence properties for the solutions of the Camassa-Holm equation, as well as unique continuation results, see also \cite{linares,freire-jpa,freire-cor}, and have been widely applied, see \cite{linares,freire-cor,henry-jnmp,him-cmp,bran-imrn,freire-jpa,raspa-mo,henri-na,zhou-jmaa,guo-siam} and references therein.

The present paper was mostly concerned with equation \eqref{1.0.7} for basically two reasons:
\begin{itemize}
    \item firstly, to complement the results established in \cite{holden-jde,bran-jde,tian-nach} looking for persistence and unique continuation results for the solutions of the equation;
    \item secondly, some equations, such as \eqref{1.0.5}, have been receiving considerable attention in recent years, but apparently nothing has been done in regard to their persistence properties. In light to these comments, by dealing with such a general model we can better understand the particular equations we studied in section \ref{sec6}.
\end{itemize}

Our theorem \ref{teo2.1} provides a unique continuation result for the solutions of Dai's equation \ref{1.0.3}, while a different proof of persistence properties of its solutions is given when compared with what was done in \cite{guo-siam}. In fact, our main inspiration for proving persistence properties for the solutions of \eqref{1.0.7} is the work by Brandolese, which is based on a general framework for addressing the question using weights in $L^p-$spaces.

We apply our results to some relevant models, such as the BBM, Dai's, DGH and rCH equations, which are models deduced in the study of propagation of waves in water models, with exception of the Dai's equation, which is studied in the context of elastic rods. 

For the Dai's equation our results can only be applied under the restrictions $0\leq\gamma\leq 3$, whereas for the rCH equation the validity of our conclusions can only be assured with the restrictions $\be^2<9\gamma$ and $\Gamma=0$ in the parameters.

The restrictions mentioned in the previous paragraph are originated from the conditions we require on the functions $f$ and $g$. While we already discussed the restriction on $g$ in section \ref{sec2}, nothing has been said about $f'(0)=0$. It is time to fill this gap.

The condition $f'(0)=0$ and the fact that $u$ is bounded imply that $f'$ is locally Lipschitz. In case $f'(0)\neq0$, then \eqref{4.4.8} should be replaced by ${\cal R}(t,x)=te^{|x|/2}R(t,x)$, with $R(t,x)=O((1+|x|)^{-1/2}(\ln{(1+|x|)})^{-d})$. Therefore, \eqref{4.4.9} evaluated at $t=t_1$, would make \eqref{2.0.8} changes to 
$$
u(t_1,x)=u_0(x)\pm t_1 e^{-|x|}(\phi_{\pm}(t_1)+\epsilon_\pm(t_1,x)+e^{|x|/2}R(x)).
$$
As such, \eqref{4.4.10} is no longer true and the result of theorem \ref{teo2.2} is not valid.

The observations above show how important the restrictions on $f$ and $g$ are and, in fact, they are sharp, in the sense that they cannot be relaxed to improve or even obtain the same results (except in the situation described below). In line with fundamental ingredients for our demonstrations, the invariance of the $H^1(\R)-$norm of the solutions is equally essential in the way we proceeded to establish our results, as we can infer from the demonstrations of theorems \ref{teo2.1}--\ref{teo2.3}.

We close this section with the following remark: If we take a careful look at \eqref{2.0.1}, we would then note that if $u=u(t,x)$ is a solution of that equation, then $v(t,x)=u(t,-x)$ would be a solution of
\bb\label{6.0.1}
u_t-f'(u)u_x-\p_x\Lambda^{-2}\Big(g(u)+\f{f''(u)}{2}u_x^2\Big)=0.
\ee

Moreover, a careful analysis of our demonstrations shows that what is really crucial is that the non-local term is non-negative. Moreover, equation \eqref{6.0.1} can be recovered from \eqref{2.0.1} by replacing $f\mapsto-f$ and $g\mapsto-g$. Therefore, our results would still be true if instead of $f''(x)\geq0$ and $g(x)\geq 0$ in ${\cal H}_1$ and ${\cal H}_2$, we assume $f''(x)\leq0$ and $g(x)\leq 0$. In this case, for consistence, condition \eqref{2.0.5} should be replaced by $|g(u(t,x))|\leq cu(t,x)^2$ (which would imply the original version again).

\section{Conclusion}\label{sec7}

In this paper we studied unique continuation and persistence properties of the solutions of the non-linear equation \ref{1.0.7}, and our main results are given in section \ref{sec2}. We then apply our results to some relevant equations describing wave propagation in physical media, mostly of them in shallow water models. 

\section*{Conflict of interest statement}
The author declares that he has no known competing financial interests or personal relationships that could have appeared to influence the work reported in this paper.

\section*{Data availability statement}

The author declares that data sharing is not applicable to this article as no data sets were generated or analysed during
the current study.

\section*{Acknowledgements}

I am grateful to CNPq (grant nº 310074/2021-5) and FAPESP (grant nº 2020/02055-0) for financial support. I am also thankful to the Institute of Advanced Studies of the Loughborough University for warm hospitality and support for my visit. In addition, it is a pleasure to thank the Mathematical Institute of the Silesian University in Opava for the nice environment, where the manuscript was finalised.

\end{document}